\documentclass[10pt]{article}
\usepackage{amsthm,amsmath,amssymb,enumerate}
\usepackage{mathrsfs}
\newtheorem{thm}{Theorem}

\newtheorem{lem}[thm]{Lemma}

\theoremstyle{definition}
\newtheorem*{rem}{Remark}

\newcommand{\s}{\sigma}
\newcommand{\la}{\lambda}

\begin{document}
\title{A note on [D. G. Higman 1988 : Strongly regular designs and coherent configurations of type
  {$[{3\atop {\;}}\;\;{2\atop 3}]$}]}
\author{Akihide Hanaki}
\date{Department of Mathematics, Faculty of Science,\\
  Shinshu University, Matsumoto 390-8621, Japan\\
  e-mail : hanaki@shinshu-u.ac.jp}
\maketitle

\begin{abstract}
  In 1988, D. G. Higman gave properties of parameters of strongly regular designs without proof.
  It included an error.
  We will give proofs and correct the error.
\end{abstract}

In 1988, D. G. Higman gave properties of parameters of strongly regular designs without proof in \cite{Higman1988}.
He wrote ``The details, which are routine, will be omitted here''.
In 2014 at Villanova, 
Mikhail Klin said to the author that Higman's paper included an error.
We will give proofs and correct the error.

This short note was first written in 2016.
The author would like to submit this
by a request of Mikhail Klin and Alyssa Sankey. 

We will consider the following fifteen equations in \cite[page 414]{Higman1988}.
\begin{enumerate}[(1)]
  \item $f_1=f_2$,
  \item $n_1 S_2=n_2 S_1$,
  \item $P_1(n_1-S_1)=(k_1-N_1)S_1$,
  \item $a_2 k_1=N_1 S_2$,
  \item $b_2\ell_1=(S_1-N_1-1)S_2$,
  \item ${N_1}^2+P_1(k_1-N_1)=k_1+\lambda_1 N_1+\mu_1(S_1-N_1-1)$,
  \item $N_1 P_1+P_1(k_1-P_1)=\lambda_1 P_1+\mu_1(S_1-P_1)$,
  \item $N_1 a_2+P_1(S_2-a_2)=S_2+a_2\lambda_1+b_2(k_1-\lambda_1-1)$,
  \item $N_1 b_2+P_1(S_2-b_2)=a_2\mu_1+b_2(k_1-\mu_1)$,
  \item $S_1+a_1 N_2+b_1(S_2-N_2-1)=S_2+a_2 N_1+b_2(S_1-N_1-1)$,
  \item $a_1 P_2+b_1(S_2-P_2)=a_2 P_1+b_2(S_1-P_1)$,
  \item $P_1(k_2-N_2)=P_2(k_1-N_1)$,
  \item $S_1+a_1 k_2+b_1\ell_2=S_2+a_2 k_1+b_2 \ell_1=S_1 S_2$,
  \item $S_1+a_1 r_2-b_1(r_2+1)=S_2+a_2 r_1-b_2 (r_1+1)$,
  \item $S_1+a_1 s_2-b_1(s_2+1)=S_2+a_2 s_1-b_2 (s_1+1)=0$.\footnote{%
  In \cite{Higman1988}, (15) is ``$S_1+a_1 S_2-b_1(s_2+1)=S_2+a_2 S_1-b_2 (s_1+1)=0$'' and this is incorrect.}
\end{enumerate}
We use the same notation in \cite{Higman1988} or \cite{Higman1987}.
However, the notation $f_i$ was used redundantly, relations and multiplicities, in \cite{Higman1988}.
We use $f_i$ for multiplicities and $\s_i$ for relations.
We also use $\s_i$ for the adjacency matrices of the relations.

\begin{lem}
  Equations (1) and (2) hold.
\end{lem}

\begin{proof}
  Since $|\mathscr{I}^{14}|=2$, we may assume that $f_1=f_2$.
  Counting $1$'s in $\s_7$, we have (2).
\end{proof}

\section{Table of multiplications}\label{sec:tbl}
It is not difficult to make a table of multiplications.
{\small
$$\begin{array}{|l||l|l|l|}
    \hline
    &\s_1&\s_2&\s_3\\
    \hline
    \hline
    \s_1 &\s_1&\s_2&\s_3\\
    \hline
    \s_2& \s_2&k_1\s_1+\la_1\s_2+\mu_1\s_3&(k_1-\la_1-1)\s_2+(k_1-\mu_1)\s_3\\
    \hline
    \s_3&\s_3&(k_1-\la_1-1)\s_2+(k_1-\mu_1)\s_3
              &(n_1-k_1-1)\s_1+(n_1-2k_1+\la_1)\s_2+(n_1-2k_1+\mu_1-2)\s_3\\
    \hline
  \end{array}$$
    
$$\begin{array}{|l||l|l|}
    \hline
    &\s_7&\s_8\\
    \hline
    \hline
    \s_1 &\s_7&\s_8\\
    \hline
    \s_2& N_1\s_7+P_1\s_8& (k_1-N_1)\s_7+(k_1-P_1)\s_8\\
    \hline
    \s_3&(S_1-N_1-1)\s_7+(S_1-P_1)\s_8&(n_1-S_1-k_1+N_1)\s_7+(n_1-S_1-k_1+P_1-1)\s_8\\
    \hline
  \end{array}$$

  $$\begin{array}{|l||l|l|l|}
    \hline
    &\s_4&\s_5&\s_6\\
    \hline
    \hline
    \s_4 &\s_4&\s_5&\s_6\\
    \hline
    \s_5& \s_5&k_2\s_4+\la_2\s_5+\mu_2\s_6&(k_2-\la_2-1)\s_5+(k_2-\mu_2)\s_6\\
    \hline
    \s_6&\s_6&(k_2-\la_2-1)\s_5+(k_2-\mu_2)\s_6
              &(n_2-k_2-1)\s_4+(n_2-2k_2+\la_2)\s_5+(n_2-2k_2+\mu_2-2)\s_6\\
    \hline
  \end{array}$$
    
$$\begin{array}{|l||l|l|}
    \hline
    &\s_9&\s_{10}\\
    \hline
    \hline
    \s_4 &\s_9&\s_{10}\\
    \hline
    \s_5& N_2\s_9+P_2\s_{10} & (k_2-N_2)\s_9+(k_2-P_2)\s_{10}\\
    \hline
    \s_6&(S_2-N_2-1)\s_9+(S_2-P_2)\s_{10}&(n_2-k_2-S_2+N_2)\s_9+(n_2-k_2-S_2+P_2-1)\s_{10}\\
    \hline
  \end{array}$$
  
$$\begin{array}{|l||l|l|l|}
    \hline
    &\s_4&\s_5&\s_6\\
    \hline
    \hline
    \s_7 &\s_7&N_2\s_7+P_2\s_8&(S_2-N_2-1)\s_7+(S_2-P_2)\s_8\\
    \hline
    \s_8& \s_8&(k_2-N_2)\s_7+(k_2-P_2)\s_8&(n_2-k_2-S_2+N_2)\s_7+(n_2-k_2-S_2+P_2-1)\s_8\\
    \hline
  \end{array}$$
    
$$\begin{array}{|l||l|l|}
    \hline
    &\s_9&\s_{10}\\
    \hline
    \hline
    \s_7 &S_2\s_1+a_2\s_2+b_2\s_3 & (S_2-a_2)\s_2+(S_2-b_2)\s_3\\
    \hline
    \s_8& (S_2-a_2)\s_2+(S_2-b_2)\s_3&(n_2-S_2)\s_1+(n_2+a_2)\s_2+(n_2+b_2)\s_3\\
    \hline
  \end{array}$$

$$\begin{array}{|l||l|l|l|}
    \hline
    &\s_1&\s_2&\s_3\\
    \hline
    \hline
    \s_9&\s_9&N_1\s_9+P_1\s_{10}&(S_1-N_1-1)\s_9+(S_1-P_1)\s_{10}\\
    \hline
    \s_{10}&\s_{10}&(k_1-N_1)\s_9+(k_1-P_1)\s_{10}&(n_1-k_1-S_1+N_1)\s_9+(n_1-k_1-S_1+P_1-1)\s_{10}\\
    \hline
  \end{array}$$
    
$$\begin{array}{|l||l|l|}
    \hline
    &\s_7&\s_{8}\\
    \hline
    \hline
    \s_9 &S_1\s_4+a_1\s_5+b_1\s_6&(S_1-a_1)\s_5+(S_1-b_1)\s_6\\
    \hline
    \s_{10}&(S_1-a_1)\s_5+(S_1-b_1)\s_6 &(n_1-S_1)\s_4+(n_1+a_1)\s_5+(n_1+b_1)\s_6\\
    \hline
  \end{array}$$
}
Other products are zero.
We can read all intersection numbers $p_{ij}^k$ by the table.

Valencies are as follows.
$$\begin{array}{lcl|lcl}
    v_1 &=& 1 & v_7 &=& S_2 \\
    v_2 &=& k_1 & v_8 &=& n_2-S_2 \\
    v_3 &=& \ell_1=n_1-k_1-1 &&\\
    \hline
    v_9 &=& S_1 & v_4 &=& 1\\
    v_{10} &=& n_1-S_1 & v_5 &=& k_2 \\
        && & v_6 &=& \ell_2=n_2-k_2-1\\
  \end{array}$$

\section{By intersection numbers}\label{sec:intnum}
By \cite[3.6 (vi)]{Higman1988}, we have
$$p_{ij}^k v_k=p_{kj^*}^{i}v_i.$$

\begin{lem}
  Equations (3), (4), and (5) hold.
\end{lem}

\begin{proof}
  The equation (3) holds by $p_{92}^{10}v_{10}=p_{10,2}^9 v_9$,
  (4) by $p_{79}^2v_2=p_{27}^7v_7$, and (5) by $p_{79}^3v_3=p_{37}^7v_7$.
\end{proof}

\section{By regular representations}\label{sec:regrep}
Let $M$ be a right regular representation of the adjacency algebra,
namely $M:\s_i\mapsto (p_{si}^t)_{st}$.\footnote{%
  In \cite{Higman1988}, $M_i^a$ is considered. It is a part of the regular representation
  omitting zero parts.}
We can read intersection numbers $p_{st}^u$ from the table of multiplications in \S \ref{sec:tbl}.

\begin{lem}
  Equations (6), (7), (8), (9), (10), (11), and (12) hold.
\end{lem}

\begin{proof}
  We consider the equation $M(\s_i)M(\s_j)=\sum_{k=1}^9 p_{ij}^k M(\s_k)$.
  The equation (6) is from $(M(\s_2)M(\s_2))_{99}$,
  (7) from $(M(\s_2)M(\s_2))_{9,10}$,
  (8) from $(M(\s_7)M(\s_9))_{22}$,
  (9) from $(M(\s_7)M(\s_9))_{23}$,
  (10) from $(M(\s_7)M(\s_9))_{99}$, 
  (11) from $(M(\s_7)M(\s_9))_{9,10}$, and
  (12) is from $(M(\s_9)M(\s_2))_{59}$.
\end{proof}

\section{By characters}\label{sec:char}
By definition, $\{\s_1,\s_2,\s_3\}$ forms an $(n_1,k_1;\lambda_1,\mu_1)$-SRG,
and $\{\s_4,\s_5,\s_6\}$ forms an $(n_2,k_2;\lambda_2,\mu_2)$-SRG.
The character tables are
$$
\begin{array}{c|ccc|c}
  &&&&\\
  \hline
  \chi_i&1&k_i&\ell_i&1\\
  \varphi_i&1&r_i&-1-r_i&f_i\\
  \psi_i&1&s_i&-1-s_i&g_i
\end{array}
$$
for $i=1,2$.
We are assuming that $f_1=f_2$ in (1).
Then the character table of the coherent configuration is
$$
\begin{array}{c||ccc|ccc||c}
  &\s_1&\s_2&\s_3&\s_4&\s_5&s_6&\\
  \hline
  \chi=\chi_1+\chi_2&1&k_1&\ell_1&1&k_2&\ell_2&1\\
  \varphi=\varphi_1+\varphi_2&1&r_1&-1-r_1&1&r_2&-1-r_2&f_1=f_2\\
  \psi_1&1&s_1&-1-s_1&0&0&0&g_1\\
  \psi_2&0&0&0&1&s_2&-1-s_2&g_2
\end{array}
$$

\begin{lem}
  Equations (13), (14), and (15) hold.
\end{lem}

\begin{proof}
  The equation (13) is by $\chi(\s_7\s_9)=\chi(\s_9\s_7)$,
  (14) by $\varphi(\s_7\s_9)=\varphi(\s_9\s_7)$,
  and (15) is by $\psi_1(\s_7\s_9)=\psi_1(\s_9\s_7)=0$ and
  $\psi_2(\s_9\s_7)=\psi_2(\s_7\s_9)=0$.
\end{proof}

\section{Remarks}

\begin{rem}
  In \cite[page 414, line 18]{Higman1988}, $p_{75}^5$ must be $p_{75}^7$.
\end{rem}

\begin{rem}
  In \cite[page 416, Table 1]{Higman1988}, we must exchange $\mu_i$ and $\lambda_i$.
\end{rem}

\bibliographystyle{amsplain}
\providecommand{\bysame}{\leavevmode\hbox to3em{\hrulefill}\thinspace}
\providecommand{\MR}{\relax\ifhmode\unskip\space\fi MR }
\providecommand{\MRhref}[2]{%
  \href{http://www.ams.org/mathscinet-getitem?mr=#1}{#2}
}
\providecommand{\href}[2]{#2}

\end{document}